\documentclass[a4paper]{article}
\pdfoutput=1
\usepackage[utf8]{inputenc}

\usepackage{lmodern}
\usepackage[round]{natbib}
\usepackage{array}
\usepackage{graphicx}
\usepackage{stmaryrd}
\usepackage{subfig}
\usepackage{amssymb, amsfonts, amsthm, amsmath}
\usepackage{enumerate}
\usepackage{enumitem}
\usepackage{bbm}

\usepackage[english]{babel}

\usepackage{tikz}
\usetikzlibrary{decorations, calc}

\usetikzlibrary{decorations.pathreplacing}

\tikzset{individu/.style={draw,thick}}

\numberwithin{equation}{section}

\theoremstyle{plain}
\newtheorem{theorem}{Theorem}[section]

\newtheorem{proposition}[theorem]{Proposition}

\theoremstyle{definition}
\newtheorem{definition}[theorem]{Definition}

\theoremstyle{remark}
\newtheorem{remark}[theorem]{Remark}

\newcommand{\calS}{\mathcal{S}}

\renewcommand{\tilde}[1]{\widetilde{#1}}
\renewcommand{\epsilon}{\varepsilon}
\renewcommand{\phi}{\varphi}

\newcommand{\Addresses}{{
  \bigskip
  \footnotesize

  \textsc{D\'epartement de Math\'ematiques et Applications, \'Ecole normale sup\'erieure, CNRS, PSL University, 45 rue d'Ulm, 75005 Paris, France}\par\nopagebreak
  \textit{E-mail address}: \texttt{sanjay.ramassamy at ens.fr}

}}

\title{The Foata correspondence, cycle lengths and anomalies}
\author{Sanjay Ramassamy}
\date{\today}

\begin{document}

\maketitle

\begin{abstract}
In their study of the densest jammed configurations for theater models, Krapivsky and Luck observe that two classes of permutations have the same cardinalities and ask for a bijection between them. In this note we show that the Foata correspondence provides the desired bijection.
\end{abstract}

\cite{krapivskyluck} introduced the theater model as a variant of directed random sequential adsorption, where spectators sequentially select a seat in a row of $L$ seats, with the constraint that they cannot go past a cluster of $b$ or more consecutive occupied seats. Configurations where all the seats are eventually occupied are induced by permutations $s$ of $\{1,\ldots,L\}$ such that for any $i$ between $1$ and $L$,
one cannot find $b$ consecutive integers $j+1,\ldots,j+b$ with $j+b<i$ and $s(j+k)>s(i)$ for all $k$ between $1$ and $b$. \cite{krapivskyluck} showed that the number $D_L^{(b)}$ of such permutations satisfies a linear recurrence relation which implies that they have the same cardinality as the permutations of $L$ elements with cycles of lengths at most $b$. The authors then asked for a bijective proof of this fact.  The goal of this note is to show that the Foata correspondence (\cite{foata,lothaire}) provides such a bijection. Interestingly enough, the Foata correspondence is already visible in~\cite{renyi}, where it is used to explain the equality of the distributions of records and cycle lengths for permutations. The present note broadens the connection between generalized notions of records and cycle lengths.

In Section~\ref{sec:Foata} we recall the Foata correspondence and in Section~\ref{sec:proof} we show that it provides the desired bijection.

\section{The Foata correspondence}
\label{sec:Foata}

Let $\calS_L$ be the group of permutations of $\{1,\ldots,L\}$. We will represent permutations in $\calS_L$ by words with $L$ distinct letters in $\{1,\ldots,L\}$. Our running example will be $359724681$, which denotes the permutation $s\in\calS_9$ such that $s(1)=3$, $s(2)=5$, $s(3)=9$, etc. The point diagram of a permutation is a plot of the graph of the corresponding function from $\{1,\ldots,L\}$ to itself, see Figure~\ref{fig:diagram} for the point diagram of the above example.

One can associate to every permutation in $\calS_L$ its cycle decomposition. Including the fixed points in that decomposition, the above $s\in\calS_9$ has cycle decomposition $[139][25][476][8]$. This way of writing is however not unique for two reasons:
\begin{itemize}
 \item each cycle of length $d$ can be written in $d$ different ways (one can freely choose what element to put first) ;
 \item if a permutation has $k$ cycles (including singletons corresponding to fixed points) one can have them appear in $k!$ different orders.
\end{itemize}
The Foata correspondence (\cite{foata,lothaire}) describes a canonical choice for writing such a cycle decomposition. Firstly we write every cycle by starting by its maximal element. We call the maximal element of a cycle the \emph{cycle head}. In the above example, $[139]$ is written as $[913]$, $[25]$ is written as $[52]$ and $[476]$ is written as $[764]$. Secondly, we write the cycles in increasing order of their cycle heads. In the above example, we obtain $[52][764][8][913]$. Removing the brackets, we obtain the word $527648913$ which can be seen as a permutation. The Foata correspondence associates to any permutation $s\in\calS_L$ the permutation $F(s)$ obtained by writing the cycle decomposition of $s$ in the above way and removing the brackets.

\begin{figure}[htbp]
\centering
\includegraphics[width=2.3in]{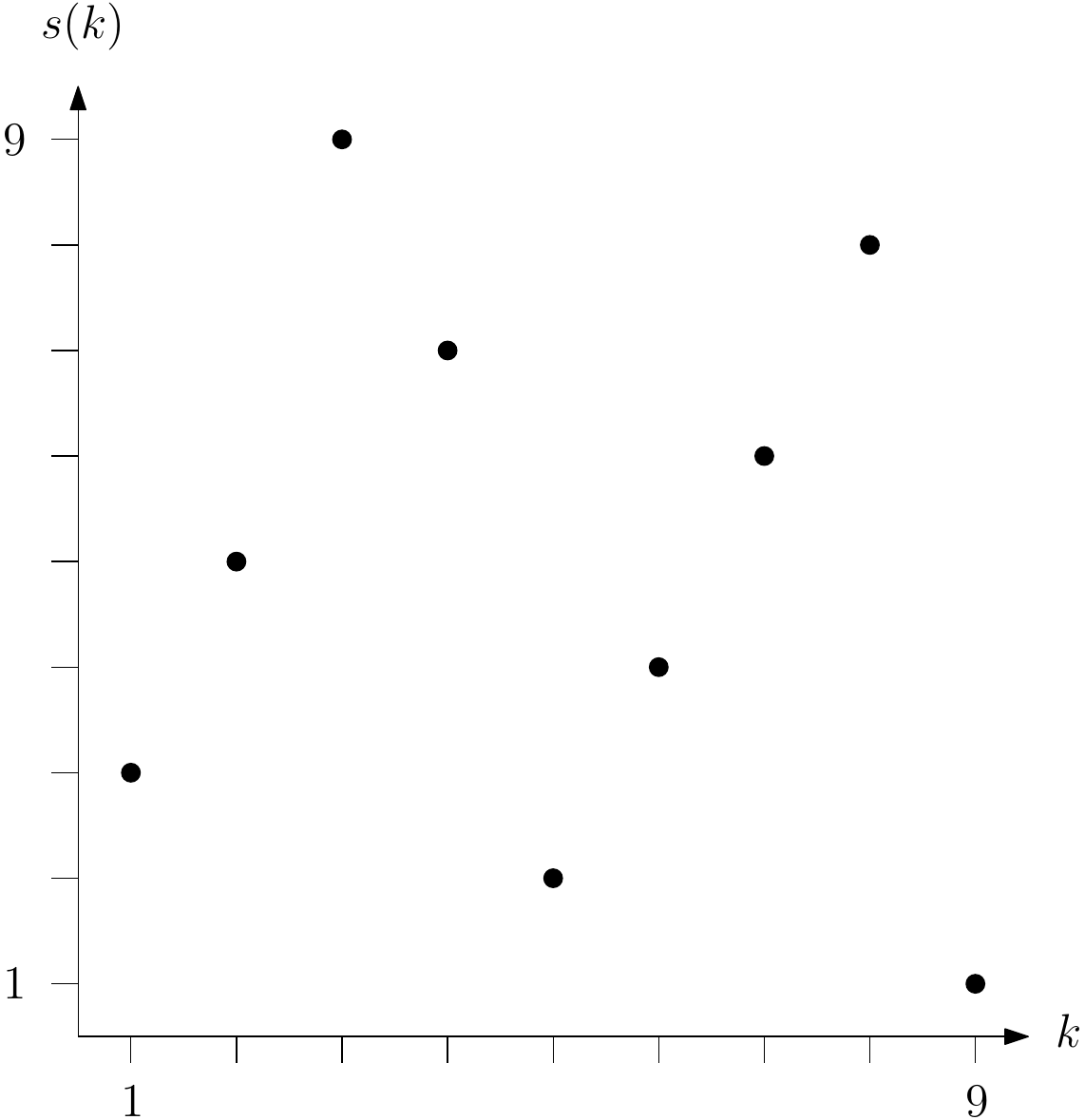}
\caption{The point diagram of the permutation $359724681$. Its cycle decomposition may be written as $[139][25][476][8]$. Writing each cycle starting by its largest element (the cycle head) and ordering the cycles in increasing order of the cycle heads, we obtain $[52][764][8][913]$. Removing the brackets, we obtain the permutation $527648913$, which is the image of $359724681$ under the Foata correspondence.}
\label{fig:diagram}
\end{figure}

\begin{remark}
\label{rem:cyclehead}
The largest letter to the left (or at the position) of a given letter $a$ in $F(s)$ corresponds to the cycle head of the cycle to which $a$ belongs in the cycle decomposition of $s$. This observation will be used later.
\end{remark}

\section{Cycle lengths and $b$-anomalies}
\label{sec:proof}

\begin{definition}
Let $b\geq1$ be an integer. Let $a_1\cdots a_L$ denote a permutation in $\calS_L$, with $L\geq1$. A consecutive subword $a_{i+1}\cdots a_{i+b}$ is called a \emph{$b$-anomaly} if there exists $1 \leq j \leq i$ such that $a_j>\max(a_{i+1},\ldots,a_{i+b})$.
\end{definition}

In terms of the point diagram, a $b$-anomaly corresponds to $b$ points with consecutive abscissae for which one can find a point strictly above and to the left of all the $b$ points. In the example of Figure~\ref{fig:diagram}, $246$ is a $3$-anomaly because $9$ is to its left and greater than $2$, $4$ and $6$.

The following result relates the cycle lengths of a permutation $s$ to the $b$-anomalies of its image $F(s)$ under the Foata correspondence.

\begin{proposition}
\label{prop:main}
Let $b\geq1$, $L\geq1$ and $s\in\calS_L$. Then $s$ has a cycle of length at least $b+1$ if and only if $F(s)$ has a $b$-anomaly.
\end{proposition}

\begin{proof}
Assume $s$ has a cycle of length $d\geq b+1$. We write it $[c_1c_2\cdots c_d]$ with $c_1$ being the cycle head, that is, the largest element of the cycle. Then the subword $c_2\cdots c_d$ of $F(s)$ forms a $(d-1)$-anomaly. Any consecutive subword of length $b$ of this $(d-1)$-anomaly provides a $b$-anomaly for $F(s)$.

Conversely, assume $F(s)=a_1\cdots a_L$ has a $b$-anomaly $a_{i+1}\cdots a_{i+b}$. By definition of the $b$-anomaly, the set
\[
X_i^s:=\left\{j \leq i | a_j> \max(a_{i+1},\ldots,a_{i+b})\right\}
\]
is non-empty, so $\max_{j\in X_i^s}{a_j}$ is well-defined and equal to some $a_h$. Then by Remark~\ref{rem:cyclehead}, in the cycle decomposition of $s$, $a_h$ is the head of the cycle to which each $a_{i+k}$ with $1 \leq k \leq b$ belongs, so there are at least $b+1$ elements in that cycle in $s$.
\end{proof}

As a consequence of Proposition~\ref{prop:main}, in order to obtain the bijection requested by~\cite{krapivskyluck}, it suffices to compose the Foata correspondence with the involution sending every permutation $s\in\calS_L$ to $\tilde{s}\in\calS_L$ defined by $\tilde{s}(i)=L+1-s(L+1-i)$ for every $i$, whereby the point diagram is rotated by $180$ degrees.

\section*{Acknowledgements}

The author thanks Jean-Marc Luck for introducing him to the theater model, the Institut de Physique Th\'eorique for the hospitality during several visits and the Fondation Sciences Math\'ematiques de Paris for the support.

\label{Bibliography}
\bibliographystyle{plainnat}
\bibliography{bibliographie}
\Addresses
\end{document}